\documentclass[reqno,12pt]{amsproc}
  \usepackage{latexsym}
  \usepackage[all]{xy}
  \usepackage{amsfonts}
  \usepackage{amsthm}
  \usepackage{amsmath}
  \usepackage{amssymb}
  \usepackage{pifont}

  \def\sw#1{{\sb{(#1)}}}

  \def\proof{{\sl Proof.~~}}

  \def\endproof{\hbox{$\sqcup$}\llap{\hbox{$\sqcap$}}\medskip}
  \def\<{{\langle}}
  \def\>{{\rangle}}

  \def\eps{\varepsilon}

  \def\note#1{{}}

\def\flip{{\rm \sf{\; flip\; }}}

  \def\note#1{}

  \def\beq{\begin{equation}}
  \def\eeq{\end{equation}}

  \def\id{\mathrm{id}}

  \def\ot{{\otimes}}

     \def\1{\mathbf{1}}

\def\act{\!\cdot\!}

\def\rcosmash{\!\!\blacktriangleright\!\!<\!}
\def\lsmash{\!>\!\!\!\triangleleft}

\def\k{\Bbbk}

  \newcounter{zlist}
  \newenvironment{zlist}{\begin{list}{(\arabic{zlist})}{
  \usecounter{zlist}\leftmargin2.5em\labelwidth2em\labelsep0.5em
  \topsep0.6ex
  \parsep0.3ex plus0.2ex minus0.1ex}}{\end{list}}

  \newcounter{blist}
  \newenvironment{blist}{\begin{list}{(\alph{blist})}{
  \usecounter{blist}\leftmargin2.5em\labelwidth2em\labelsep0.5em
  \topsep0.6ex 
  \parsep0.3ex plus0.2ex minus0.1ex}}{\end{list}}

  \newcounter{rlist}

   \newcounter{alist}

  \renewcommand{\subjclassname}{\textup{2010} Mathematics Subject
        Classification}

\def\stac#1{\raise-.2cm\hbox{$\stackrel{\displaystyle\otimes}{\scriptscriptstyle{#1}}$}}

\def\cten#1{\raise-.2cm\hbox{$\stackrel{\displaystyle\widehat{\otimes}}
{\scriptscriptstyle{#1}}$}}

  \def\Label#1{\label{#1}\ifmmode\llap{[#1] }\else
  \marginpar{\smash{\hbox{\tiny [#1]}}}\fi}
  \def\Label{\label}

  \newtheorem{proposition}{Proposition}[section]
  \newtheorem{lemma}[proposition]{Lemma}
  
  \newtheorem{theorem}[proposition]{Theorem}

  \theoremstyle{definition}
  \newtheorem{definition}[proposition]{Definition}

  \newtheorem{example}[proposition]{Example}

  \theoremstyle{remark}
  \newtheorem{remark}[proposition]{Remark}

  \theoremstyle{definition}
  \swapnumbers

\begin{document}

 \title{$R$-smash products of Hopf quasigroups}
 \author{Tomasz Brzezi\'nski}
 \address{{\bf Tomasz Brzezi\'nski: }Department of Mathematics, Swansea University,
  Singleton Park,
  Swansea SA2 8PP, U.K.}
  \email{T.Brzezinski@swansea.ac.uk}
  
 \author{Zhengming Jiao}
 \address{{\bf Zhengming Jiao: } Department of Mathematics,
Henan Normal University,
Xinxiang 453002, Henan,
Peoples Republic of China } 
 \email{zmjiao@henannu.edu.cn}

\date{May 2010}

  \begin{abstract}
The theory of $R$-smash products for Hopf quasigroups is developed.\\

\noindent{\bf \subjclassname :} {16T05; 16T15}\\
{\bf Keywords}: Hopf quasigroup; $R$-smash product.
  \end{abstract}
  \maketitle

\section{Introduction}
Hopf quasigroups and Hopf coquasigroups were introduced in \cite{KliMaj:Hop}. These are non-associative (or non-coassociative) generalisations of Hopf algebras, in which the antipode provides one with a certain level of control over the non-associativity. In particular Hopf quasigroups are examples of unital coassociative $H$-bialgebras introduced in \cite[Section~2]{Per:alg}. They can be understood as linearisations of loops \cite{Alb:qua}.  Also in \cite{KliMaj:Hop} smash products of Hopf quasigroups were studied. It has been shown in \cite{BrzJiao:Act} that a standard form of a smash product forces one to replace the conventional associativity of action (assumed in  \cite{KliMaj:Hop} from the onset) by a similar condition involving the antipode. In this note, which is a sequel to \cite{BrzJiao:Act}, we look at $R$-smash products \cite{CaeIon:Sma} of Hopf quasigroups and, briefly, at $W$-smash coproducts of Hopf coquasigroups. This analysis reveals that some of the conventional requirements on the twisting map $R$ need be replaced by similar conditions in which the antipode plays a prominent role; see Theorem~\ref{def.RsmashP} and Definition~\ref{def.allterms} for details.

All algebras and coalgebras are over a field $\k$ and they are assumed to be unital and counital respectively, but are not assumed to be associative or coassociative unless stated atherwise. Unadorned tensor product symbol represents the tensor product of $\k$-vector spaces. We use the standard Sweedler notation for coproducts $\Delta(h) = h\sw 1\ot h\sw 2$ (summation understood) even if the coproduct $\Delta$ is not assumed to be associative.

\section{$R$-smash products of Hopf quasigroups}\label{sec.R-smash}
 \setcounter{equation}{0}
We begin by introducing terminology used in this note. 

\begin{definition}\label{def.allterms}
Let $H$ be an algebra with product $\mu_H$ and unit $1_H$ and a coalgebra with coproduct $\Delta_H$ and counit $\eps_H$ that are algebra morphisms. Similarly, let $A$ be an algebra with product $\mu_A$ and unit $1_A$ and a coalgebra with coproduct $\Delta_A$ and counit $\eps_A$ that are algebra morphisms. Consider linear maps $S_H: H\to H$ and $R:H\ot A\to A\ot H$. The map $R$ is said to be:
\begin{itemize}
\item {\em left normal} (resp.\ {\em right normal}) if 
$$
R\circ( \id_H \ot 1_A) = 1_A \ot \id_H, \qquad (\mbox{resp. } R\circ (1_H\ot \id_A) = \id_A\ot 1_H), 
$$
and it is said to be {\em normal} if it is both left and right normal;
\item {\em left multiplicative} if
$$
R\circ (\id_H\ot \mu_A) = (\mu_A\ot \id_H)\circ (\id_A \ot R)\circ (R\ot\id_A);
$$
\item {\em right $S_H$-multiplicative} if
$$
R\circ (\mu_H \ot \id_A)\circ (\id_H\ot S_H\ot \id_A) = (A\ot \mu_H)\circ (R\ot \id_H)\circ (\id_H\ot R)\circ (\id_H\ot S_H\ot \id_A);
$$
\item {\em right $S_H$-normal} if
$$
R\circ (S_H\ot \id_A)\circ \flip \circ R\circ (1_H\ot \id_A) = 1_H\ot \id_A.
$$
\end{itemize}
Dually, the map $R$ is said to be:
\begin{itemize}
 \item {\em left conormal} (resp.\ {\em right conormal}) if 
$$
( \eps_A \ot \id_H)\circ R = \id_H\ot \eps_A, \qquad (\mbox{resp. } (\id_A\ot \eps_H)\circ  R= \eps_H\ot \id_A), 
$$
and it is said to be {\em conormal} if it is both left and right conormal;
\item {\em left comultiplicative} if
$$
(\Delta_A\ot \id_H)\circ R = (\id_A\ot R)\circ (R\ot \id_A)\circ (\id_H \ot \Delta_A);
$$
\item {\em right $S_H$-comultiplicative} if
$$
(\id_A\ot S_H\ot \id_H)\circ (\id_A\ot \Delta_H)\circ R = (\id_A\ot S_H\ot \id_H)\circ (R\ot \id_H)\circ(\id_H\ot R)\circ (\Delta_H\ot A);
$$
\item {\em right $S_H$-conormal} if
$$
(\id_A\ot \eps_H)\circ R\circ \flip \circ (\id_A\ot S_H)\circ R = \eps_H\ot \id_A.
$$
\end{itemize}
\end{definition}

The action of $R$ on elements is denoted by
$$
R(h\ot a) = \sum_R a_R\ot h^R = \sum_r a_r\ot h^r, \quad \mbox{etc.},
$$
for all $h\in H$ and $a\in A$. The reader is encouraged to write down all the above requirements on $R$ in terms of this notation. For example, $R$ is left multiplicative if 
\begin{equation}\label{LM}
\sum_R (ab)_R \ot h^R = \sum_{R,r} a_R b_r \ot h^{Rr},
\end{equation}
and is right $S$-multiplicative if
\begin{equation}\label{SRM}
\sum_R a_R\ot (gS(h))^R = \sum_{R,r} a_{Rr}\ot g^r S(h)^R,
\end{equation}
for all $a,b\in A$ and $g,h\in H$, etc.

We are particularly interested in the case in which $H$ and $A$ are Hopf quasigroups or Hopf coquasigroups, and $S_H$ is the antipode of $H$. We will concentrate on the former case, as the latter can be treated dually. Recall from \cite{KliMaj:Hop} that $(H,\mu_H,1_H,\Delta_H,\eps_H,S_H)$ as in Definition~\ref{def.allterms} is called a {\em Hopf quasigroup} provided $\Delta_H$ is coassociative and the following {\em Hopf quasigroup identities} are fulfilled
\begin{equation}\label{quasi1}
S_H(h\sw 1)(h\sw 2g)  = g\eps(h)= h\sw 1(S_H(h\sw 2) g) ,
\end{equation}
\begin{equation}\label{quasi2}
(gh\sw 1)S_H(h\sw 2) = g\eps(h)= (gS_H(h\sw 1))h\sw 2 .
\end{equation}
for all $g,h\in H$. The  identities \eqref{quasi1}--\eqref{quasi2} ensure that a Hopf quasigroup is  an {\em $H$-bialgebra} with left division $h\setminus g = S_H(h)g$ and right division $g/h = gS_H(h)$; see \cite[Definition~2]{Per:alg}.  It is proven in \cite{KliMaj:Hop} that the antipode $S_H$ is antimultiplicative and anticomultiplicative and it immediately follows from the Hopf quasigroup identities that $S_H$ enjoys the standard antipode property.

\begin{definition}\label{def.Rsmash} 
Let $H$ and $A$ be Hopf quasigroups, $R: H\ot A\to A\ot H$ a $\k$-linear map. An {\em $R$-smash product} of $H$ and $A$ is a  Hopf quasigroup $A\lsmash_R  H$  equal to $A\ot H$ as a vector space, with  tensor product coproduct, unit and counit, and the multiplication 
\begin{equation}\label{Rsmashmulti}
\mu = (\mu_A\ot \mu_H)\circ (\id_A\ot R\ot \id_H) 
\end{equation}
and antipode 
\begin{equation}\label{Rsmashanti}
S = R\circ (S_H\ot S_A)\circ \flip .
\end{equation}
\end{definition}

The aim of this note is to determine necessary and sufficient conditions for $R$ to produce an $R$-smash product of Hopf quasigroups. These are listed in the following theorem, which is a Hopf quasigroup version of \cite[Corollary~4.6]{CaeIon:Sma}

\begin{theorem}\label{def.RsmashP}
Let $H, A$ be Hopf quasigroups, $R: H\ot A\to A\ot H$ a $\k$-linear map. If $R$ is left multiplicative and left conormal, then the following statements are equivalent:
\begin{zlist}
\item
$A\lsmash_R  H$ is an $R$-smash product Hopf quasigroup for $H$ and $A$;
\item 
The map $R$ is a coalgebra map that is normal, right $S_H$-multiplicative and right $S_H$-conormal.
\end{zlist}
\end{theorem}

Before the proof of Theorem~\ref{def.RsmashP} is given we state and prove three lemmata.

\begin{lemma}\label{lemma.con-coa}
Let $H, A$ be Hopf quasigroups, $R: H\ot A\to A\ot H$ a $\k$-linear map. If $R$ is a left conormal coalgebra map, then,
for all $h\in H, a\in A$, 
\begin{equation}\label{coalg1}
\sum_Ra_R\ot h^R\sw 1\ot h^R \sw 2=\sum_Ra_R\ot h\sw 1^R\ot h\sw 2,
=\sum_Ra_R\ot h\sw 1\ot h \sw 2^R,
\end{equation} 
hence
\begin{equation}\label{coalg3}
R(h\ot a) = \sum_Ra_R\eps_H(h\sw 1^R)\ot h\sw 2 = \sum_Ra_R\ot h\sw 1\eps_H(h \sw 2^R).
\end{equation}
Furthermore, $R$ is left comultiplicative.
\end{lemma}
\begin{proof}
Equations \eqref{coalg1} follow by appying $\id_A \ot \id_H \ot \eps_A\ot \id_H$ or $\eps_A \ot \id_H \ot \id_A\ot \id_H$ to the formula expressing the comultiplictivity of $R$,  i.e.\ to
\begin{equation}\label{coalgmap}
\sum_R a_R\sw 1 \ot h^R\sw 1 \ot a_R\sw 2 \ot h^R\sw 2 = \sum_{R,r}  a\sw 1_R \ot h\sw 1^R\ot a\sw 2_r \ot h\sw 2^r,
\end{equation}
and by using the left conormality of $R$. Equations \eqref{coalg3} then follow from \eqref{coalg1} by applying $\id_A\ot \eps_H\ot \id_H$ and $\id_A\ot \id_H\ot \eps_H$. 

Finally, apply $\id_A\ot \eps_H \ot \id_A\ot\id_H$ to \eqref{coalgmap} and use \eqref{coalg3} to compute
$$
\sum_R a_R\sw 1\ot a_R\sw 2 \ot h^R = \sum_{R,r} a\sw 1_R\eps_H(h\sw 1^R) \ot a\sw 2_r \ot h\sw 2^r = \sum_{R,r} = a\sw 1_R\ot a\sw 2_r\ot h^{Rr}.
$$
Thus, $R$ is left comultiplicative as required.
\end{proof}

 \begin{lemma}\label{lem.quasimult}
 Let $H, A$ be Hopf quasigroups, $R: H\ot A\to A\ot H$ a $\k$-linear map. If $R$ is a left conormal coalgebra map, then:
 \begin{zlist}
 \item $R$ is right $S_H$-multiplicative if and only if, for all $a\in A$, $g,h\in H$,
 \begin{equation}\label{WRM}
\sum_{R,r}a_{Rr}\varepsilon_H (g^rS_H(h)^R)=\sum_R a_R\varepsilon_H ((g_HS(h))^R).
\end{equation}
 \item For all $a\in A$, $h\in H$, the conditions
 \begin{equation}\label{QRM}
\sum_{R,r} a_{Rr}\varepsilon_H (S_H(h_{(1)})^rh\sw{2}^R)=\varepsilon_H (h)a=\sum_{R,r}a_{Rr}\varepsilon_H (h\sw{1}^rS_H(h\sw{2})^R)
\end{equation}
are equivalent to 
\[\sum_{R,r} a_{Rr}\ot S_H(h_{(1)})^rh\sw{2}^R=a\ot h =\sum_{R,r}a_{Rr}\ot h\sw{1}^rS_H(h\sw{2})^R .
\]
\item If $R$ is right normal, then $R$ is right $S_H$-multiplicative and right $S_H$-conormal if and only if it satisfies \eqref{WRM} and \eqref{QRM}.
\end{zlist}
 \end{lemma}
 \begin{proof}
 (1) Obviously, the right $S_H$-multiplicativity of $R$ implies  \eqref{WRM}. Conversely, the right $S_H$-multiplicativity  of $R$ can be inferred from \eqref{WRM} by repetitive use of equations \eqref{coalg3} in Lemma~\ref{lemma.con-coa}:
 \begin{align*}
 \sum_{R} a_{R} \ot  (gS_H(h))^R & \stackrel{\eqref{coalg3}}{=} \sum_R a_R \eps_H((gS_H(h))\sw 1^R) \ot (gS_H(h))\sw 2 \\
 &= \sum_R a_R \eps_H((g\sw 1S_H(h\sw 2))^R) \ot g\sw 2S_H(h\sw 1)\\
 & \stackrel{\eqref{WRM}}{=} \sum_{R,r} a_{Rr} \eps_H(g\sw 1^r)\eps_H(S_H(h\sw 2)^R) \ot g\sw 2S_H(h\sw 1)\\
 & \stackrel{\eqref{coalg3}}{=} \sum_{R,r} a_{Rr} \eps_H(g\sw 1^r)\ot g\sw 2S_H(h)^R \stackrel{\eqref{coalg3}}{=} \sum_{R,r} a_{Rr} \ot g^rS_H(h)^R ,
 \end{align*}
 as required. 
 
 The second statement  is proven by a similar repetitive use of equations \eqref{coalg3} in Lemma~\ref{lemma.con-coa}, and the proof is left to the reader. To prove (3), if $R$ is right normal and right $S_H$-multiplicative, then
\begin{align*}
 \sum_{R,r}a_{Rr}\varepsilon_H  (h\sw{1}^rS_H(h\sw{2})^R)  \stackrel{\eqref{WRM}}{=} & \sum_R a_R \eps_H((h\sw 1 S_H(h\sw 2))^R)\\
  = &\sum_R a_R\eps_H(h)\eps_H({1_H}^R) = a\eps_H(h),
 \end{align*}
 so the second of equations \eqref{QRM} is automatically satisfied. Now, we need to use the multiplicativity of the counit and \eqref{coalg3} in Lemma~\ref{lemma.con-coa} to compute
 \begin{align*}
 \sum_{R,r} a_{Rr}\varepsilon_H (S_H(h_{(1)})^rh\sw{2}^R) & = \sum_{R,r} a_{Rr} \eps_H(h\sw 2^R)\eps_H(S_H(h_{(1)})^r) = \sum_{R,r} a_{Rr}\eps_H(S_H(h^R)^r).
 \end{align*}
 Hence, the first of equations \eqref{QRM} is equivalent to right $S_H$-conormality of $R$.
 \end{proof}

\begin{lemma}\label{def.RSA}
Let $H$ and $A$ be Hopf quasigroups, $R: H\ot A\to A\ot H$ a left normal and left multiplicative map. If $R$ is also a coalgebra map and is left conormal, then
\begin{equation}\label{RSA}
R\circ (\id_H \ot S_A) = (S_A\ot \id_H)\circ R. 
\end{equation} 
Furthermore, the first of equalities \eqref{QRM} implies that
\begin{equation}\label{S.cae}
R \circ \flip \circ (S_A\ot S_H)\circ R \circ \flip = S_A \ot S_H.
\end{equation}
\end{lemma}

\proof Take any $a\in A$ and $h\in H$. Then, using the left multiplicativity and left conormality of $R$ to make a start and to finish, we can compute
\begin{align*}
\sum_R S_A(a_R) \ot h^R &= \sum_{R,\bar{R}, r} S_A(a\sw 1_R)(a\sw 2 _{\bar{R}}S_A(a\sw 3)_r)\ot h^{R\bar{R}r}\\
& \stackrel{\eqref{coalgmap}}{=} \sum_{R, r} S_A(a\sw 1_R\sw 1)(a\sw 1 _{R}\sw 2S_A(a\sw 2)_r)\ot h^{Rr}
 \stackrel{\eqref{quasi1}}{=} \sum_r S_A(a)_r \ot h^r.
 \end{align*} 
This proves equality \eqref{RSA}.

The second assertion is proven by the following calculation, for all $a\in A$, $h\in H$,
\begin{align*}
\sum_{R, r} S_A(a_R)_r& \ot S_H(h^R)^r  \stackrel{\eqref{RSA}}{=}\sum_{R,r}S_A(a_{Rr}) \ot S_H(h^R)^r \\
& \stackrel{\eqref{coalg3}}{=}\sum_{R,r}S_A(a_{Rr})\eps_H(S_H(h\sw 1)\sw 1^r h\sw 2^R) \ot S_H(h\sw 1)\sw 2 \\
&= \sum_{R,r}S_A(a_{Rr})\eps_H(S_H(h\sw 2)^r h\sw 3^R) \ot S_H(h\sw 1) \stackrel{\eqref{QRM}}{=} S_A(a)\ot S_H(h).
\end{align*}
Thus the equality \eqref{S.cae} holds as required.
 \endproof

{\sl Proof of Theorem~\ref{def.RsmashP}.} 
 $(2)\Rightarrow (1)$ The normality of $R$ immediately implies  that $1_A\ot 1_H$ is the unit of $A\lsmash_R  H$. The left counitality of $R$ together with the fact that a counit of a Hopf quasigroup  is an algebra map imply that also the counit 
 $\varepsilon_A\ot \varepsilon_H$ of $A\lsmash_RH$  is an algebra homomorphism. The coproduct  $\Delta$  of  $A\lsmash_R  H$ is obviously unital, and it is multiplicative since $R$ is a coalgebra morphism. This part of the proof is not different from the standard Hopf algebra case; see \cite{CaeIon:Sma}.
It remains to check the Hopf quasigroup identities \eqref{quasi1} and \eqref{quasi2}, which is done by explicit calculations. 

 For all $a,b\in A$ and $g,h\in H$, 
\begin{align*}
S((a\ot h)_{(1)}) ((a\ot h)_{(2)}(b\ot g))
& \stackrel{\eqref{Rsmashanti},\eqref{Rsmashmulti}}{=}
\sum_{{\bar R}, R,r} {S_A(a_{(1)})_{R}}(a_{(2)}{b_r})_{\bar R}\ot S_H(h_{(1)})^{R{\bar R}}(h\sw{2}^rg)\\
&\stackrel{\eqref{LM}}{=}\sum_{R,r} (S_A(a_{(1)})(a_{(2)}{b_r}))_R\ot S_H(h_{(1)})^{R}(h\sw{2}^rg)\\
&\stackrel{\eqref{quasi1}}{=}\sum_{R,r} \eps_A(a)b_{rR}\ot S_H(h_{(1)})^{R}(h\sw{2}^rg)\\
&\stackrel{\eqref{coalg3}}{=}\sum_{R,r} \eps_A(a)b_{rR}\eps_H(S_H(h_{(2)})^Rh\sw{3}^r)\ot S_H(h_{(1)})(h_{(4)}g)\\
&\stackrel{\eqref{QRM}}{=}\varepsilon_A (a)b\ot S_H(h_{(1)})(h_{(2)}g)
\stackrel{\eqref{quasi1}}{=}\varepsilon_A (a)\varepsilon_H (h)b\ot g.
\end{align*}
This proves the first of equations \eqref{quasi1}. Next
\begin{align*}
(a\ot h)_{(1)}(S((a\ot & h)_{(2)})(b\ot g))
\stackrel{\eqref{Rsmashanti},\eqref{LM}}{=}
\sum_{R,r,{\bar R}} a_{(1)}(S_A(a\sw 2)_Rb_r)_{\bar R} \ot h\sw 1^{\bar R} (S_H(h\sw 2)^{Rr}g)\\
&\stackrel{\eqref{LM}}{=} 
\sum_{R,{\bar R}} a_{(1)}(S_A(a\sw 2)b)_{R{\bar R}} \ot h\sw 1^{\bar R} (S_H(h\sw 2)^{R}g)\\
&\stackrel{\eqref{coalg3}}{=}
\sum_{R,{\bar R}} a_{(1)}(S_A(a\sw 2)b)_{R{\bar R}} \eps_H(h\sw 1^{\bar R} (S_H(h\sw 4)^{R}) \ot h\sw 2 (S_H(h\sw 3)g)\\
&\stackrel{\eqref{quasi1}}{=} \sum_{R,{\bar R}} a_{(1)}(S_A(a\sw 2)b)_{R{\bar R}} \eps_H(h\sw 1^{\bar R} (S_H(h\sw 2)^{R}) \ot g\\
&\stackrel{\eqref{QRM}}{=}\varepsilon_H (h)a_{(1)}(S_A(a_{(2)})b)\ot g
\stackrel{\eqref{quasi1}}{=}\varepsilon_A (a)\varepsilon_H (h)b\ot g,
\end{align*}
where also the normality was used to derive the penultimate euqality. This proves the second of relations \eqref{quasi1}. It is the proof of  \eqref{quasi2} where the right $S_H$-multiplicativity of $R$, \eqref{SRM},  is used. The first of identities  \eqref{quasi2}  is proven by the following calculation
\begin{align*}
((b\ot g)(a\ot h)_{(1)})S((a\ot &h)_{(2)})
\stackrel{\eqref{Rsmashmulti},\eqref{Rsmashanti}}{=}
\sum_{R,r,{\bar R}} (ba\sw 1_R)S_A(a\sw 2)_{r{\bar R}} \ot(g^Rh\sw 1)^{\bar R}S_H(h\sw 2)^r\\
&\stackrel{\eqref{SRM}}{=}\sum_{R,r} (ba\sw 1_R)S_A(a\sw 2)_{r} \ot((g^Rh\sw 1)S_H(h\sw 2))^r\\
&\stackrel{\eqref{quasi2}}{=}
\sum_{R,r} \eps_H(h) (ba\sw 1_R)S_A(a\sw 2)_{r} \ot g^{Rr}\\
&\stackrel{\eqref{RSA}}{=}
\sum_{R,r} \eps_H(h) (ba\sw 1_R)S_A(a\sw 2_{r}) \ot g^{Rr}\\
&\stackrel{}{=}\sum_{R} \eps_H(h) (ba_R\sw 1)S_A(a_R\sw 2) \ot g^{R}
\stackrel{\eqref{quasi2}}{=}\varepsilon_A (a)\varepsilon_H (h)b\ot g.
\end{align*}
The penultimate equality  is a consequence of left comultiplicativity of $R$ which is asserted by Lemma~\ref{lemma.con-coa}. In derivation of the last equality, the left conormality of $R$ was also used.
Finally, and again using Lemma~\ref{lemma.con-coa} in the penultimate equality and left conormality of $R$ in the last one, we compute
\begin{align*}
((b\ot g)S((a\ot h)_{(1)}))(a\ot & h)_{(2)}
\stackrel{\eqref{Rsmashanti},\eqref{Rsmashmulti}}{=} 
\sum_{R,r,{\bar R}} (b S_A(a\sw 1)_{Rr})a\sw 2_{\bar R} \ot (g^r S_H(h\sw 1)^R)^{\bar R}h\sw 2\\
&\stackrel{\eqref{SRM}}{=} 
\sum_{R,{\bar R}} (b S_A(a\sw 1)_{R})a\sw 2_{\bar R} \ot (g S_H(h\sw 1))^{R\bar R}h\sw 2\\
&\stackrel{\eqref{RSA}}{=} 
\sum_{R,{\bar R}} (b S_A(a\sw 1_{R}))a\sw 2_{\bar R} \ot (g S_H(h\sw 1))^{R\bar R}h\sw 2\\
&\stackrel{}{=} 
\sum_{R} (b S_A(a_R\sw 1))a_R\sw 2 \ot (g S_H(h\sw 1))^{R}h\sw 2
\stackrel{\eqref{quasi2}}{=} \varepsilon_A (a)\varepsilon_H (h)b\ot g.
\end{align*}
This completes the proof that  $A\lsmash_R  H$ is a Hopf quasigroup as required.

 $(1)\Rightarrow (2)$
 The fact that $1_A\ot 1_H$ is the unit of the $R$-smash product Hopf quasigroup $A\lsmash_R  H$ immediately implies the normality of $R$.
 The equalities, for all $a\in A$, $h\in H$,
$$
\Delta ((1_A\ot h)(a\ot 1_H))=\Delta (1_A\ot h)\Delta (a\ot 1_H), \quad \eps((1_A\ot h)(a\ot 1_H))=\eps(1_A\ot h)\eps (a\ot 1_H),
$$
resulting from the multiplicativity of coproduct and counit imply that $R$ is a coalgebra map. This is no different from the standard Hopf algebra case.

Developing  the second of the Hopf quasigroup conditions \eqref{quasi2} for $A\lsmash_R  H$ as in the first part of the proof of the theorem, and using Lemma~\ref{def.RSA} and \eqref{coalg3} repeatedly, we arrive at the following equality:
$$
\sum_{R,r,{\bar R}} (b S_A(a\sw 1_{Rr}))a\sw 2_{\bar R}\eps_H(g\sw 1^r S_H(h\sw 1)\sw 1^R) \ot (g\sw 2 S_H(h\sw 1)\sw 2)^{\bar R}h\sw 2 = \eps_A(a)\eps_H(h) b\ot g.
$$
Apply $\id_A \ot \varepsilon_H$ to both sides of this equation, set $b=\sum_{{\hat R}, {\hat r}}a\sw 1_{{\hat R} {\hat r}}\eps_H(g\sw 1^{\hat r} S_H(h)\sw 1)^{\hat R})$,  and shift Sweedler's indices as required to obtain:
\begin{align*}
\sum_{R,r,{\bar R},{\hat R}, {\hat r}} ( a\sw 1_{{\hat R} {\hat r}}\eps_H(g\sw 1^{\hat r} S_H(h)\sw 1^{\hat R}) S_A(a\sw 2_{Rr}))a\sw 3_{\bar R}\eps_H(g\sw 2^r &S_H(h)\sw 2^R)  \eps_H( (g\sw 3 S_H(h)\sw 3)^{\bar R})\\
& = \sum_{{\hat R}, {\hat r}} a_{{\hat R} {\hat r}}\eps_H(g^{\hat r} S_H(h))^{\hat R}).
\end{align*}
The fact that $R$ is a  coalgebra map, equation~\eqref{coalgmap}, implies 
\begin{align*}
\sum_{R,r,{\bar R}} ( a\sw 1_{{R} {r}}\sw 1 S_A(a\sw 1_{Rr}\sw 2))a\sw 2_{\bar R}\eps_H(g\sw 1^r S_H(h)\sw 1^R) & \eps_H((g\sw 2 S_H(h)\sw 2)^{\bar R})\\
& = \sum_{{\hat R}, {\hat r}} a_{{\hat R} {\hat r}}\eps_H(g^{\hat r} S_H(h))^{\hat R}).
\end{align*}
Finally,  the antipode property combined with the left conormality  of $R$ yield equation \eqref{WRM}. Therefore, $R$ is right $S_H$-multiplicative by Lemma~\ref{lem.quasimult}.

Finally, making the same steps in the proof of the first Hopf quasigroup  identity for $A\lsmash_R  H$ as in the first part of the proof of the theorem, we conclude that the first of conditions \eqref{quasi1} imply
$$
\sum_{R,r} \eps_A(a)b_{rR}\eps_H(S_H(h_{(2)})^Rh\sw{3}^r)\ot S_H(h_{(1)})(h_{(4)}g) = \eps_A(a)\eps_H(h) b\ot g.
$$ 
Applying $\id_A\ot \eps_H$ and evaluating this identity at $a=1_A$ and $g=1_H$ one immediately obtains  the first of equations \eqref{QRM}. In view of Lemma~\ref{lem.quasimult} this is tantamount to right $S_H$-conormality of $R$.
\endproof

\begin{remark} \label{rem.cae}
In \cite[Corollary~4.6]{CaeIon:Sma}, which is a predecessor of Theorem~\ref{def.RsmashP}, it is assumed that the $R$ is compatible with the antipodes so that the equality \eqref{S.cae} is satisfied. As explained in Lemma~\ref{def.RSA} this follows from other hypotheses made in Theorem~\ref{def.RsmashP}, most notably from right $S_H$- and left conormality of $R$, which are not assumed in \cite{CaeIon:Sma}.
\end{remark}

\begin{example}
Let $H$ and $A$ be Hopf quasigroups. Recall from \cite{BrzJiao:Act}  that  $A$ is a {\em left $H$-quasimodule Hopf quasigroup}, if 
\begin{blist}
\item $A$ is a left $H$-quasimodule, i.e.\ there exists a $\k$-linear map $H\ot M\rightarrow M$, $h\ot m \mapsto h\act m$, such that, for all $a\in M$ and $h\in H$,
$$
1_H\act m=m, \qquad 
h_{(1)}\act (S_H(h_{(2)})\act m)=\varepsilon_H (h)m=S_H(h_{(1)})\act (h_{(2)}\act m);
$$
\item the $H$-action satisfies the following compatibility conditions:
$$
(h_{(1)}\act a)(h_{(2)}\act b)=h\act(ab),  \qquad h\act 1_A=\varepsilon_H (h)1_A,
$$
$$
\Delta_A (h\act a)=h_{(1)}\act a_{(1)}\ot h_{(2)}\act a_{(2)}, \qquad \varepsilon_A (h\act a)=\varepsilon_H (h)\varepsilon_A (a),
$$
for all $h\in H, a, b \in A$. 
\end{blist}
Let $A$ be a left $H$-quasimodule Hopf quasigroup such that,  for all $g,h\in H$ and $a\in A$, 
\begin{equation}\label{quasi.ass}
h_{(1)}\ot h_{(2)}\act a=h_{(2)}\ot h_{(1)}\act a, \qquad g\act (S_H(h)\act a) = (gS_H(h))\act a .
\end{equation}
Then the map 
$$
R: H\ot A\to H\ot A, \qquad h\ot a\mapsto h\sw 1\act a\ot h\sw 2,
$$
is a coalgebra map which is left multiplicative, normal, left conormal, right $S_H$-multi\-plicative and right $S_H$-conormal. Consequently, there is an $R$-smash product Hopf quasigroup for $H$ and $A$, $A\lsmash_R  H$. $A\lsmash_R  H$ coincides with the smash product $A\lsmash  H$ described in \cite{BrzJiao:Act}.

That $R$ satisfies all assumptions of Theorem~\ref{def.RsmashP} can be checked by straightforward calculations which are left to the reader. We only mention in passing that for the right $S_H$-conormality of $R$ both equations \eqref{quasi.ass} are needed.
\end{example}

Dually to a Hopf quasigroup, an algebra and coalgebra $H$ with coproduct and counit that are algebra maps is called a {\em Hopf coquasigroup} provided the product  is associative and there exists a linear map $S_H: H\to H$ such that, for all $h\in H$,
$$
S_H(h\sw 1)h\sw 2\sw 1\ot h\sw 2\sw 2   = 1_H\ot h = h\sw 1S_H(h\sw 2\sw 1)\ot h\sw 2\sw 2
$$
and
$$
h\sw 1\sw 1 \ot h\sw 1\sw 2 S_H(h\sw 2)  = h\ot 1_H = h\sw 1\sw 1 \ot S_H(h\sw 1\sw 2) h\sw 2.
$$
When written in terms of commutative diagrams, the definitions of a Hopf quasigroup and Hopf coquasigroup are formally dual to each other in the sense that one is obtained by reversing all the arrows in the definition of the other. Thus the theory of $R$-smash coproducts for Hopf coquasigroups can be obtained by dualising the theory of $R$-smash products for Hopf quasigroups. By this means one can first state

\begin{definition}\label{def.Wsmash} 
Let $H$ and $A$ be Hopf coquasigroups and let $W: H\ot A\to A\ot H$ be a $\k$-linear map. By a {\em $W$-smash coproduct} of $H$ and $A$ we mean a Hopf coquasigroup  $H{}_W\rcosmash  A$ that is equal to $H\ot A$ as a vector space with tensor product unit, multiplication and counit, and with comultiplication and antipode 
$$
\Delta = (\id_H\ot W\ot \id_A)\circ (\Delta_H\ot \Delta_A), \qquad
S=\flip \circ (S_A \ot S_H)\circ W.
$$
\end{definition}

Then, dualising Theorem~\ref{def.RsmashP}, we obtain the following Hopf coquasigroup version of \cite[Corollary~4.8]{CaeIon:Sma}
\begin{theorem}
Let $H, A$ be Hopf coquasigroups, $W: H\ot A\to A\ot H$ a $\k$-linear map. If $W$ is left comultiplicative and left normal, then the following statements are equivalent:
\begin{zlist}
\item
$H_W\rcosmash  A$ is a $W$-smash coproduct Hopf coquasigroup of $H$ and $A$;
\item 
The map $W$ is an algebra map that is conormal, right $S_H$-comultiplicative and right $S_H$-normal.
\end{zlist}
\end{theorem}

 \section*{Acknowledgements} 
Zhengming Jiao would like to thank the members of Department of Mathematics at Swansea University for hospitality. His research is supported by the Natural Science Foundation of Henan Province (grant no.\ 102300410049). Tomasz Brzezi\'nski would like to express his deep gratitude to the organisers of the First International Conference on Mathematics and Statistics (AUS-ICMSÕ10), American University of Sharjah, for wonderful hospitality. He  would also like to thank the members of Department of Mathematics at the University of New Brunswick in Fredericton, where this note was completed, for hospitality. His visit to Fredericton was supported by the European Commission grant PIRSES-GA-2008-230836 hereby gratefully acknowledged.

 \end{document}